\documentclass[11pt,a4paper,reqno]{amsart}

\usepackage{tikz}
\usetikzlibrary{decorations.pathreplacing,shapes,arrows}
\newcommand{\arrowdraw}[2]{\draw [postaction={decorate}] (#1)--(#2);}

\usetikzlibrary{decorations.markings}
\usetikzlibrary{arrows,matrix}
\usetikzlibrary{snakes}
\usepgflibrary{arrows}
\tikzset{
    >=stealth',
    punkt/.style={
           rectangle,
           rounded corners,
           draw=black, very thick,
           text width=6.5em,
           minimum height=2em,
           text centered},
    pil/.style={
           ->,
           thick,
           shorten <=2pt,
           shorten >=2pt,}
}
\tikzset{every loop/.style={min distance=2mm,in=225,out=135,looseness=10}}

\usepackage{tikz}
\usetikzlibrary{decorations.markings}
\usetikzlibrary{arrows,matrix}
\usepgflibrary{arrows}

\newcommand\GreenL{\mathcal{L}}
\newcommand\GreenR{\mathcal{R}}

\usepackage{clrscode}
\usepackage{mathrsfs,amssymb}

\usepackage{amsmath, amsthm}

\newtheorem{theorem}{Theorem}[section]

\newtheorem{question}[theorem]{Question}
\newtheorem{lemma}[theorem]{Lemma}
\newtheorem{corollary}[theorem]{Corollary}
\newtheorem{conjecture}[theorem]{Conjecture}
\newtheorem{proposition}[theorem]{Proposition}

\begin{document}

\title[Amenability and Geometry of Semigroups]{Amenability and Geometry of Semigroups}

\maketitle

\begin{center}
    ROBERT D. GRAY\footnote{School of Mathematics, University of East Anglia, Norwich NR4 7TJ, England.
Email \texttt{Robert.D.Gray@uea.ac.uk}.}
\ and \ MARK KAMBITES\footnote{School of Mathematics, University of Manchester, Manchester M13 9PL, England. Email \texttt{Mark.Kambites@manchester.ac.uk}.} \\
\end{center}

\begin{abstract}
We study the connection between amenability, F\o lner conditions and the geometry of finitely generated semigroups. Using
results of Klawe, we show that within an extremely broad class of semigroups (encompassing all groups, left
cancellative semigroups, finite semigroups, compact topological semigroups, inverse semigroups, regular semigroups, commutative semigroups and semigroups with a left, right or two-sided zero element), left amenability coincides with the strong F\o lner condition. Within the same class, we show that a finitely generated semigroup of subexponential growth is left amenable if and only if it is left reversible.
We show that the (weak) F\o lner condition is a left quasi-isometry invariant of finitely generated semigroups, and hence that left amenability is a left quasi-isometry invariant of left cancellative semigroups. We also give a new characterisation of the strong F\o lner condition, in terms of the existence of \textit{weak} F\o lner sets satisfying a local injectivity condition on the relevant translation action of the semigroup. 
\end{abstract}

\section{Introduction}\label{sec_intro}

What are now called amenable groups were introduced in 1929 by von Neumann
\cite{vonNeumann29}, motivated by the desire for a group-theoretic
understanding of paradoxical decompositions such as the Banach-Tarski
paradox. The term ``amenable'' was coined by Day \cite{Day57}, who also
broadened consideration to encompass semigroups. In the decades that
followed, amenability --- in groups,
semigroups and also Banach algebras --- has developed into a major topic
of study, forming a remarkably deep vein of connections
between different areas of mathematics including algebra, analysis, geometry,
combinatorics and dynamics; see \cite{Paterson88} for a comprehensive
introduction and for example \cite{Cecc10} for more recent developments.
Important recent research on amenable semigroups, specifically, includes
for example \cite{Cecc15,Donnelly13,Willson09},

In the context of amenability, groups and semigroups have many similarities,
with many of the key results being shared.
A notable exception, however, appears where finitely generated objects are concerned. Within the study of amenable groups, there is a distinct strand of research focussing on amenability of \textit{finitely generated} groups, and hence linking the subject to combinatorial and geometry group theory. Notable results include the quasi-isometry invariance of
amenability (see \cite[Theorem 10.23]{Ghys89}) and
Grigorchuk's beautiful characterisation
of amenability in terms of
\textit{cogrowth} \cite{Grigorchuk80}.
In contrast, despite widespread interest in finitely generated semigroups both within algebra and from theoretical computer science, amenability of finitely generated semigroups has yet to receive a comparable
level of attention. 
One reason for this is that results about amenability for finitely generated groups are largely built on elementary characterisations in terms of \textit{F\o lner sets} which tend to manifest themselves in natural ways in Cayley graphs, and hence lend themselves to study using the tools and techniques of geometric group theory.  This poses two problems in the semigroup setting.

The first is that the theory of finitely generated semigroups has in 
recent decades been more combinatorial and less geometric than that of 
finitely generated groups: the ideal structure of a semigroup is 
notoriously difficult to capture in a geometric way, leading researchers 
to prefer other tools such as rewriting systems and automata. However, 
recent advances such as the authors' use of \textit{asymmetric geometry}
\cite{K_svarc, K_qsifp, K_semimetric, K_hyper}  seem to be 
bringing a genuinely geometric aspect to the subject.

The second problem is that amenability in semigroups is not characterised by the
F\o lner set property, or indeed by any known elementary combinatorial condition.  To be precise,
the F\o lner condition admits two formulations which are trivially equivalent
in a group (or indeed, a left cancellative semigroup), but not in a general
semigroup.
The weaker of these (generally known as ``the F\o lner condition'' or \textit{FC}) was
studied by Day \cite{Day57} who showed that it was a necessary but not a
sufficient condition for left amenability. The stronger form (``the
strong F\o lner condition'' or \textit{SFC}) was considered by Argabright and Wilde
\cite{Argabright67}; they showed that SFC is sufficient for left amenability,
and that necessity would follow from a conjecture of Sorenson \cite{Sorenson64,Sorenson66}, asserting
that every right cancellative, left amenable semigroup was also left
cancellative. Subsequently, Klawe \cite{Klawe77}
showed that necessity was actually equivalent
to Sorenson's conjecture, before disproving both, by producing a (non-finitely-generated)
example of a left amenable semigroup which was right but not left cancellative; this was subsequently
refined to a finitely generated example by Takahashi \cite{Takahashi03}. Despite considerable further
work in this area (see for example \cite{Yang87}), an elementary combinatorial
characterisation of (left or two-sided) amenability for semigroups remains elusive.

While SFC does not provide an exact characterisation of left amenability
for semigroups in full generality, it can do so within large and important
classes of semigroups. We have already noted that in the presence of left
cancellativity SFC is trivially equivalent to FC (which in general is weaker
that left amenability), and so both FC and SFC characterise left amenability
within the class of left cancellative semigroups. In fact, we shall see below
(Theorem~\ref{thm_klawesfc} and Corollary~\ref{cor_idptsfc}) that SFC is
equivalent to left amenability within
an extremely large class of (not necessarily finitely generated) semigroups including not only left cancellative and commutative semigroups, but 
also semigroups in which every ideal contains an idempotent: the latter
condition encompasses
for example all groups, finite semigroups, compact left or right topological semigroups, inverse semigroups, regular semigroups and semigroups with a left, right or two-sided zero element, and the result therefore
applies to the overwhelming majority of widely studied semigroups.

We also present (in Section~\ref{sec_newsfc})  a new characterisation of SFC, in terms of the existence of \textit{weak} F\o lner sets satisfying a local injectivity condition on the relevant translation action of the semigroup. As well as having potential applications
as a technical lemma for studying SFC, this result is conceptually interesting,
because it gives a new insight
into the difference between FC and SFC, and what exactly it is about
left cancellativity which is important for amenability.


The equivalence of FC and/or SFC with left amenability in such a wide range of semigroups provides
a strong motivation for studying these conditions, and how they
relate to the geometry of finitely generated semigroups: results so obtained will bear directly upon amenability for a very large range of semigroups, and are also likely to give clues as to how amenability itself could be directly studied for finitely generated semigroups in even greater generality. The chief aim of the present paper is to begin this study.

One of our main results (Theorem~\ref{thm_polyklawe}) is that, for semigroups in the same broad class described
above, a finitely generated semigroup of subexponential growth has SFC (and hence is left amenable) if and only if it is left reversible (the latter being a trivially necessary condition for amenability in all semigroups).
We also show (Theorem~\ref{thm_qsi}) that FC is a left quasi-isometry
invariant of finitely generated semigroups, and hence that left amenability
is a left quasi-isometry invariant of finitely generated left cancellative
semigroups. Neither amenability nor SFC is a left quasi-isometry invariant
of finitely generated semigroups more generally, but it remains open whether
these properties can be seen in the right quasi-isometry class, or in the
left and right quasi-isometry classes together.

\section{Analytic, Algebraic and Geometric Conditions}

In this section we briefly recall the definitions of left amenability, the
F\o lner condition and the strong F\o lner condition for semigroups. For
a more detailed introduction we direct the reader to the monograph
of Paterson \cite{Paterson88}. We then show that, mainly as a consequence
of work of Klawe \cite{Klawe77}, the strong F\o lner condition
exactly characterises left amenability within an extremely broad class
of semigroups.

\subsection{Amenability and F\o lner Conditions}
A semigroup $S$ is called \textit{left amenable} if there is a mean
on $l_\infty(S)$ which is invariant under the natural left action
of $S$ on the dual space $l_\infty(S)'$ \cite[Section 0.18]{Paterson88}.
Equivalently \cite[Problem 0.32]{Paterson88}, $S$ is left amenable if
it admits a finitely additive probability measure $\mu$,
defined on all the subsets of $S$, which is \textit{left invariant},
in the sense
that $\mu(a^{-1} X) = \mu(X)$ for all $X \subseteq S$ and $a \in S$. Here,
$a^{-1} X$ denotes the set $\lbrace s \in S \mid as \in X \rbrace$. (Note that
left invariance is strictly weaker than requiring
$\mu(aX) = \mu(X)$ for all $X \subseteq S$ and $a \in S$; semigroups
admitting a finitely additive probability measure satisfying this
much stronger property are called \textit{left measurable} --- see \cite{Sorenson66}
and \cite[Section~5]{Klawe77} for more on this property.)

A semigroup $S$ satisfies \textit{the F\o lner condition} (\textit{FC})
if for every finite subset $H$ of $S$ and every $\epsilon > 0$, there
is a finite subset $F$ of $S$ with $|sF \setminus F| \leq \epsilon |F|$ for
all $s \in H$.
A semigroup $S$ satisfies \textit{the strong F\o lner condition} (\textit{SFC})
if for every finite subset $H$ of $S$ and every $\epsilon > 0$, there
is a finite subset $F$ of $S$ with $|F \setminus sF| \leq \epsilon |F|$
for all $s \in H$.

In a left cancellative semigroup FC and SFC are trivially equivalent,
but without left cancellativity SFC is strictly stronger, because
left translation by an element $s$ can map many elements in a set
$F$ onto a few elements, allowing $sF \setminus F$ to be small
but $F \setminus sF$ large. For semigroups in general, it is known that
SFC implies left amenability \cite{Argabright67}, which in turn
implies FC \cite{Day57}, but neither of these implications is reversible
\cite{Day57,Klawe77}. (For left cancellative semigroups, of course, it
follows that FC and SFC both exactly characterise amenability.)

\subsection{Left Thick Subsets and Left Reversibility}

Recall that a subset $E$ of a semigroup $S$ is called \textit{left thick} if
for every finite subset $F$ of $S$ there is an element $t$ of $S$ such that
$Ft \subseteq E$. Left thickness is of interest in the theory of
amenability, because it provides an abstract algebraic characterisation of those
subsets capable of having ``full measure'': more precisely, a subset
of a left amenable semigroup $S$ is left thick if and only if it has measure $1$
in some left invariant, finitely additive probability measure $S$ \cite[Proposition~2.1]{Paterson88}.

In any semigroup, it is immediate from the definition that every left ideal (and
hence 
every two-sided ideal) is left thick. Right ideals in a general semigroup
need not be left thick, but in left amenable semigroups it transpires that
they are. Recall that a semigroup is called \textit{left reversible} if every pair
of right ideals (or equivalently, every pair of principal right ideals)
intersects.
Left reversibility is easily seen to be a necessary precondition
for left amenability \cite[Proposition 1.23]{Paterson88}.
Examples of left reversible semigroups include inverse semigroups, commutative semigroups, and cancellative semigroups which embed in groups of left quotients.
The following elementary result connects left reversibility with left thickness:
\begin{proposition}\label{prop_reversiblethick}
For any semigroup $S$, the following conditions are equivalent:
\begin{itemize}
\item[(i)] $S$ is left reversible;
\item[(ii)] every principal right ideal of $S$ is left thick;
\item[(iii)] every right ideal of $S$ is left thick;
\end{itemize}
\end{proposition}
\begin{proof}
Suppose first that (i) holds, and let $aS^1$ be a principal right ideal of
$S$. Let
$\lbrace f_1, \dots, f_n \rbrace$ be a finite subset of $S$. Since $S$ is left
reversible, the right ideal $f_1 S$ intersects the right ideal $aS$, so we may choose $t_1 \in S$ with
$f_1 t_1 \in aS$. Similarly, $f_2 t_1 S$ intersects $aS$, so we may choose
$t_2 \in S$ so that $f_2 t_1 t_2 \in aS$. Continuing in this way, we
define a sequence of elements $t_1, \cdots, t_n \in S$ so that
$f_i t_1 \dots t_i \in aS$ for each $i$. But since $aS$ is a right ideal,
it follows that $f_i t_1 \dots t_n \in aS$ for all $i$. This shows
that $aS$, and hence $a S^1$, is left thick, and hence that (ii) holds.

It is immediate from the definition that a set containing a left thick
set is left thick, so the fact that (ii) implies (iii) follows from the fact
that every right ideal contains a principal right ideal.

Finally, suppose (iii) holds, and let $aS^1$ and $bS^1$ be principal right
ideals. Since $aS^1$ is left thick and $\lbrace b \rbrace$ is a finite set,
there is a $t \in S$ so that $\lbrace b \rbrace t \subseteq aS^1$. But now
$bt \in aS^1 \cap bS^1$, as required to show that (i) holds.
\end{proof}

\subsection{Near Left Cancellativity.} In this section we introduce a
new definition which will prove useful for understanding the relationship
between amenability, F\o lner conditions and cancellativity conditions. 
We say that a semigroup $S$ is \textit{near left cancellative}
if for every element
$s$ of $S$ there is a left thick subset $E$ on which the left translation
map by $s$ restricts to an injective map, that is, such that $sx \neq sy$
whenever $x,y \in E$ with $x \neq y$.

(Near left cancellativity is stronger than that property
informally termed  \textit{almost left cancellativity} in \cite[p.104]{Klawe77},
which is that for every $s$ the set of elements $t$ such
that there is no other element $t'$ with $st = st'$ is left thick. The
term \textit{almost left cancellative} is also used in a completely
different sense in \cite[Section 7.22]{Paterson88}.)

It is immediate that a left cancellative semigroup is near
left cancellative, since the whole semigroup will always be a
left thick subset of itself.
Perhaps a more surprising observation is any semigroup with a
right zero element is near left
cancellative: indeed, the singleton set containing a right zero
is easily seen to be a left thick subset on which the left translation
action of every element of $S$ cannot help but be injective.
This may seem initially troubling to the reader versed in the algebraic
theory of semigroups, where the existence of a zero element is often 
regarded as the antithesis of cancellativity; from this perspective to declare
\textit{everything} with a zero to be close to left cancellative
may seem bizarre. However, when viewing semigroups in a 
\textit{dynamic} context the logic is clearer: intuitively, while a
semigroup with zero may have arbitrary 
algebraic complexity ``above'' the zero, random walks on the semigroup
eventually end at zero with probability $1$, meaning the asymptotic
dynamical behaviour is the same as that of the (left cancellative)
trivial monoid.



The following proposition shows that the class of near left cancellative
semigroups is very large, including for example all left reversible regular
semigroups (and hence all inverse semigroups) and all left reversible
finite semigroups.

\begin{proposition}\label{prop_idptsalc}
Let $S$ be a left reversible semigroup in which every ideal contains an
idempotent (for example, a left reversible regular and/or left reversible
finite semigroup). Then $S$ is near left cancellative.
\end{proposition}
\begin{proof}
Let $s \in S$, and consider the ideal $SsS$. Since every ideal contains
an idempotent, we may choose an idempotent $e \in SsS$, say $e = xsy$.
Now $xsyxsy = xsy$ implies $e = xsy \GreenR xsyxs$. Thus, $xsyxs$ is
a regular element and so is also $\GreenL$-related to an idempotent, say
$xsyxs \GreenL f = f^2$. In particular, $f$ is $\GreenL$-below $s$, so
we may write $f = ts$ for some $t$.

Now consider the right ideal $fS$. By Proposition~\ref{prop_reversiblethick}
this is left thick. Moreover,
for any element $fz \in fS$ we have $t(s(fz)) = (ts)(fz) = f^2 z = fz$
which tells us that left translation by $s$ is injective on $fS$
as required.
\end{proof}

Examples of semigroups which are \textit{not} near left cancellative
include non-trivial left zero semigroups, or more generally, semigroups
which are right cancellative but not left cancellative. (This claim will
be justified in the remarks following Proposition~\ref{prop_alcklawe} below.)

\subsection{The Klawe Condition and Right Cancellative Quotients}

We say that a semigroup $S$ satisfies the \textit{Klawe condition} if
whenever $s$, $x$ and $y$ in $S$ are such that $sx = sy$, there exists
$t \in S$ so that $xt = yt$. This condition was implicitly introduced and
used, although not given a name, by Klawe \cite{Klawe77}. Every left
cancellative semigroup satisfies
this condition vacuously; in fact, it transpires that
near left cancellativity suffices:
\begin{proposition}\label{prop_alcklawe}
Every near left cancellative semigroup satisfies the Klawe condition.
\end{proposition}
\begin{proof}
Let $S$ be near left cancellative. Suppose $sx = sy$. Let $E$ be a left
thick subset of $S$ on which $s$
acts injectively by left translation. Since $E$ is left thick, we may
choose a $t \in S$ such that $\lbrace x,y \rbrace t \subseteq E$. Now
$s(xt) = (sx) t = (sy) t = s(yt)$ where $xt, yt \in E$. But left translation
by $s$ is injective on $E$, so we must have $xt = yt$.
\end{proof}

Note that a right cancellative semigroup satisfying the Klawe condition
must clearly be left cancellative, so Proposition~\ref{prop_alcklawe}
justifies the claim in the previous section that a semigroup which is
right cancellative but not left cancellative cannot be near left
cancellative.

On any semigroup $S$, we may define a binary relation by $x \cong y$ if and
only if there exists an $s$ with $xs = ys$. In general this relation is not
transitive, but in the case $S$ is left reversible it is actually
a congruence, and the quotient $S / \mathord{\cong}$ is a right cancellative
semigroup \cite[Proposition~1.24]{Paterson88}. (Note that left reversibility
is not a \textit{necessary} condition for the relation to be a congruence or
the quotient to be right cancellative, as witnessed for example by a free
semigroup of rank $2$.)

We shall need the following elementary fact, one implication of which
was shown by Klawe \cite{Klawe77}.
\begin{proposition}\label{prop_klawecanc}
Let $S$ be a left reversible semigroup. Then $S$ satisfies the Klawe
condition if and only if $S / \mathord{\cong}$ is left cancellative.
\end{proposition}
\begin{proof}
The direct implication is the contrapositive of \cite[Lemma~2.1]{Klawe77}.

For the converse, suppose $S / \mathord{\cong}$ is left cancellative, and that
$sx = sy$. Then writing $[a]$ for the $\cong$-equivalence class of
$a$, we have $[s][x] = [s][y]$, so by left cancellativity of the quotient
we must have $[x] = [y]$, that is, $x \cong y$. But by definition this means that
there is a $t$ with $xt = yt$, as required.
\end{proof}
Notice that the left reversibility hypothesis is required only to ensure
that $\cong$ is a congruence so that $S / \mathord{\cong}$ is actually well-defined. 
The Klawe condition does not suffice to ensure left reversibility (consider
for example a free semigroup of rank $2$, or any cancellative semigroup
which does not embed in a group) or indeed even for $\cong$ to be transitive.
For example, the semigroup
$\langle a, b, c, x, y \mid ax = bx, by = cy \rangle$
is left cancellative, and hence trivially satisfies the Klawe condition,
but we have $a \cong b$ and $b \cong c$ but $a \not\cong c$.

Klawe \cite[Theorem~2.2]{Klawe77} showed that a left amenable semigroup satisfies
SFC if and only if the right cancellative quotient $S  / \cong$
is also left cancellative. Combining this with Proposition~\ref{prop_klawecanc}
yields a three-way equivalence:

\begin{theorem}\label{thm_klawe}
Let $S$ be a left amenable semigroup. Then the following are equivalent:
\begin{itemize}
\item[(i)] $S$ satisfies the strong F\o lner condition;
\item[(ii)] the right cancellative quotient $S / \mathord{\cong}$ is left cancellative;
\item[(iii)] $S$ satisfies the Klawe condition.
\end{itemize}
\end{theorem}

Combining with a result of Argabright and Wilde \cite{Argabright67} we obtain the fact that
for the (very large) class of semigroups satisfying the Klawe condition, the strong F\o lner condition
gives an exact characterisation of amenability.

\begin{theorem}\label{thm_klawesfc}
Let $S$ be a semigroup satisfying the Klawe condition. Then $S$ is left amenable if and only
if $S$ satisfies SFC.
\end{theorem}
\begin{proof}
One implication is immediate from Theorem~\ref{thm_klawe}; the other is
\cite[Theorem~1]{Argabright67}.
\end{proof}

\begin{corollary}\label{cor_idptsfc}
Let $S$ be a semigroup in which every ideal contains an idempotent (for example a regular semigroup,
inverse semigroup, finite semigroup, compact left or right topological semigroup, or semigroup with a left, right or two-sided zero). Then $S$ is left amenable if and only if $S$ satisfies SFC.
Moreover, if $S$ is left amenable then for every $s \in S$ there is a left invariant
finitely additive probability measure on $S$
such that the left  translation map by $s$ is injective when restricted to some set of full measure.
\end{corollary}
\begin{proof}
If $S$ is left amenable then in particular it is left reversible, so we deduce by Proposition~\ref{prop_idptsalc} that $S$ is near left cancellative, by Proposition~\ref{prop_alcklawe} that $S$ satisfies
the Klawe condition and so by Theorem~\ref{thm_klawesfc} that $S$ satisfies
SFC. The converse is again \cite[Theorem~1]{Argabright67}.

Moreover, if $s \in S$ then since $S$ is near left cancellative, there
is a left thick subset $E$ of $S$ such that the left translation map of $s$
on $E$ is injective. But since $E$ is left thick, by \cite[Proposition~1.21]{Paterson88},
there exists a left invariant finitely additive probability measure on $S$ such that $E$ has full measure.
\end{proof}

Of course, it is not the case that every left thick subset of
a left amenable semigroup can be made to have full measure simultaneously,
with respect to the same measure; consider for example a non-trivial finite
right
zero semigroup, which is left amenable but has disjoint left thick subsets.
However, with reference to the latter part of Corollary~\ref{cor_idptsfc},
one may ask whether
it is necessary for the measure to be chosen differently for different
translation maps, or if a single measure suffices:
\begin{question}
If a semigroup is left amenable and near left cancellative, is there
necessarily 
a left invariant finitely additive probability measure such that every
element acts injectively by
left translation on some set of full measure?
\end{question}


\section{A new characterisation of SFC}\label{sec_newsfc}

In this section we present a new characterisation of SFC, in terms of the existence of \textit{weak} F\o lner sets satisfying a local injectivity condition on the relevant translation action of the semigroup. As well as having potential applications as a technical lemma for studying SFC, this result is conceptually interesting, as it gives a new insight into what exactly it is about left cancellativity which is important for amenability.
The proof is a straightforward direct argument.

\begin{theorem}
Let $S$ be a semigroup. Then $S$ satisfies the strong F\o lner condition
if and only if for every finite set $F \subseteq S$ and $\epsilon > 0$ there is a finite set $X \subseteq S$ such that
for each $f \in F$ we have
$$|fX \setminus X| \leq \epsilon |X|$$ 
and for all $x,y \in X$, if $fx = fy$ then $x = y$.
\end{theorem}
\begin{proof}
Suppose the given condition is satisfied, and given $F$ and $\epsilon$ choose $X$ as in the condition. Then for each $f \in F$, the injectivity of the action of $f$ on $X$ implies that $|fX| = |X|$, whence
$$|X \setminus fX| \ = \ |fX \setminus X| \ \leq \epsilon |X|$$
as required to show that SFC holds.

Conversely, suppose $S$ satisfies SFC and let $F$ and $\epsilon$ be given. Let $\mu > 0$ be
small. By SFC we can choose a finite set $A$ such that $|fA \setminus A| < \mu |A|$. Define 
$$B \ = \ \lbrace a \in A \mid \lnot \left( \exists f \in F, b \in A, b \neq a, fb = fa \right) \rbrace$$
to be the set of all elements of $A$ which form singleton fibres under the action of left translation by
each element of $F$.
Clearly by definition, elements of $F$ acts injectively by left translation on $B$. We claim that
\begin{equation}\label{eq_claim1}
|A \setminus B| \ \leq \ 2 |F| \mu |A|
\end{equation}
Indeed, clearly we have
$$A \setminus B = \bigcup_{f \in F} \lbrace a \in A \mid \exists b \in A, b \neq a, fa = fb \rbrace.$$
We write $C_f$ for the component of the union on the right-hand side corresponding to $f \in F$.
Now if \eqref{eq_claim1} does not hold then there must be some $f \in F$ such that $|C_f| > 2 \mu |A|$.
By the definition of $C_f$, it is clear that $|f C_f| \leq \frac{1}{2} |C_f|$ so we have
\begin{align*}
|fA| &\leq |f C_f| + |f (A \setminus C_f)| \ \\
&\leq \ \frac{1}{2} |C_f| + |A \setminus C_f| \\
&= \ |A| - \frac{1}{2} |C_f| \\
&< \ |A| - \mu |A| \\
&= \ (1- \mu) |A|.
\end{align*}
But this means that
$|A \setminus fA| > \mu |A|$, contradicting the choice of $A$. This completes the proof of
\eqref{eq_claim1}. 

It follows also from \eqref{eq_claim1} that
\begin{equation}\label{eq_claim2}
|B| \ \geq \ |A| - 2|F| \mu |A| \ = \ (1-2|F| \mu) |A|
\end{equation}

Now for any $s \in F$ we have
\begin{align*}
|B \setminus sB| \ &\leq \ |A \setminus sB| &\textrm{(since $B \subseteq A$)} \\
&\leq \ |A \setminus sA| + |sA \setminus sB| \\
&\leq \ |sA \setminus A| + |A \setminus B| \\
&\leq \ \mu |A| + 2 |F| \mu |A| & (\textrm{by the definition of $A$ and } \eqref{eq_claim1})\\
&= \ (1 + 2|F|) \mu |A| \\
&\leq \ \frac{(1 + 2|F|) \mu}{1-2|F| \mu} |B| &(\textrm{by } \eqref{eq_claim2})
\end{align*}
By choosing $\mu > 0$ sufficiently small we may make
$$\frac{(1+2|F|) \mu}{1-2|F| \mu} \ < \ \epsilon$$
(since the left hand-side as a function of $\mu$ takes the value $0$ at
$\mu = 0$ and is clearly continuous away from $\mu = 1/(2|F|)$, so that
$$|B \setminus sB| \ \leq \ \epsilon |B|$$
as required.
\end{proof}

\section{Growth and Amenability}

Our aim in this section is to show that for finitely generated semigroups
satisfying the Klawe condition, sub-exponential growth is a sufficient
condition for SFC, and hence for left amenability.

Recall that if $M$ is a semigroup generated by a finite subset $X$, the
\textit{growth function} of $M$ with respect to $X$ is the function which
maps a natural number $n$ to the number of distinct elements of $M$ which can
be written as a product of $n$ or fewer generators from $X$. Although the
growth function of $M$ depends on the choice of finite generating set, its
asymptotic behaviour does not, and is an invariant of the monoid. We say
that $M$ has \textit{polynomial growth} if its growth function is bounded
above by a polynomial, or \textit{subexponential growth} if its growth
function is eventually bounded above by every increasing exponential function.
Growth of finitely generated semigroups is a major topic in both abstract semigroup
theory and application areas --- see for example \cite{Cedo07,Grigorchuk88,Shneerson08}
for work in this area.

We shall need the following lemma about semigroups satisfying the Klawe
condition.
\begin{lemma}
Let $a$ and $b$ be elements of a semigroup $S$ satisfying the Klawe
condition. Then either $aS$ and $bS$ intersect, or $a$ and $b$ freely
generate a free subsemigroup of rank $2$.
\end{lemma}
\begin{proof}
Consider the natural map
$f : \lbrace A,B \rbrace^+ \to S$ from the free semigroup on symbols
$A$ and $B$ to $S$, taking $A$ to $a$ and $B$ to $b$. If $a$ and $b$
do not freely generate a free subsemigroup then this map cannot be injective,
so we may choose distinct words $u,v \in \lbrace A,B \rbrace^+$ such
that $f(u)=f(v)$.

We may assume without loss of generality that
neither $u$ nor
$v$ is a prefix of the other. Indeed, if $u$ is a prefix of $v$, then choose
$c \in \lbrace A,B \rbrace$ such that $c$ is not the letter of $v$
immediately following the prefix $u$. Then clearly we have $f(uc) = f(vc)$
where neither $uc$ nor $vc$ is a prefix of the other, so we may simply replace
$u$ by $uc$ and $v$ by $vc$. A symmetric argument applies if $v$ is a prefix of
$u$.

Now let $w$ be the longest common prefix of $u$ and $v$, and write
$u = wu'$ and $v = wv'$. The assumption from the previous paragraph
ensures that $u'$ and $v'$ are non-empty. If $w$ is non-empty then $f(w) f(u') = f(w) f(v')$ in $S$, so
by the Klawe condition we may choose $s \in S$ such that $f(u') s = f(v') s$. If $w$ is
empty then $f(u') = f(v')$ so we may choose $s \in S$ arbitrarily to obtain the same property.
But by construction, $u'$ and $v'$ begin with different letters from
$\lbrace A,B \rbrace$, so the element $f(u')s = f(v')s$ lies in both $aS$ and $bS$.
\end{proof}

An immediate corollary is a useful fact about
semigroups satisfying the Klawe condition:
\begin{corollary}\label{cor_klawedichotomy}
Let $S$ be a semigroup satisfying the Klawe condition. Then either $S$
is left reversible or $S$ contains a free subsemigroup of rank $2$.
\end{corollary}
Note that the two possibilities in Corollary~\ref{cor_klawedichotomy}
are not mutually exclusive: for example, adjoining a zero element to a
free semigroup of rank $2$ yields a semigroup which satisfies both
conditions (and also the Klawe condition).

Our next result says that all finitely generated semigroups of
subexponential growth satisfy the weak condition FC. In itself this
is not especially interesting --- indeed, it might be thought of as
simply indicating just how weak a condition FC is --- but combined
with other results it will allow us to establish sufficient conditions
for SFC.

\begin{theorem}\label{thm_polyfc}
Let $S$ be a finitely generated semigroup of subexponential growth. Then
$S$ satisfies the F\o lner condition FC.
\end{theorem}
\begin{proof}
Given a finite subset $F$ of $S$, we may choose a finite generating
set $X$ for $S$ containing $F$. Let $B_i$ denote the
ball of radius $i$ in $S$ with respect to the generating set $X$, that
is, the set of all elements which can be written as a product of $i$ or
fewer generators from $X$. We
claim that
$$\inf_{i \in \mathbb{N}} \frac{|B_{i+1}|}{|B_i|} = 1.$$
Indeed, because $B_{i} \subseteq B_{i+1}$ for each $i$ the given
infimum is certainly at least $1$, so if the claim were false it would
be bounded below by some $\lambda > 1$. But then we have 
$|B_{i+1}| / |B_i| > \lambda$ for all $i$ and certainly $|B_1| > 0$, so
$|B_i| \geq |B_1| \lambda^{i-1}$ for all $i$,
contradicting the subexponential growth of $S$.

Hence, given $\epsilon > 0$, we may choose $i$ with $|B_{i+1}| / |B_i| < 1+\epsilon$.
For any $f \in F$ we have $f \in X$ so $f B_i \subseteq B_{i+1}$ and now
$$|f B_i \setminus B_i| \ \leq \ |B_{i+1} \setminus B_i| \ \leq \ \epsilon |B_i|$$
as required to show that $S$ satisfies FC.
\end{proof}

Since the F\o lner condition suffices for left amenability within the class
of left cancellative semigroups, it follows immediately from
Theorem~\ref{thm_polyfc} that left cancellative semigroups
of subexponential growth are left amenable. But in fact, by combining with
previously known results, we are in a position to prove something much
stronger:

\begin{theorem}\label{thm_polyklawe}
Let $S$ be a finitely generated semigroup of subexponential growth and
satisfying the Klawe condition. Then $S$ is left amenable and satisfies
the strong F\o lner condition.
\end{theorem}
\begin{proof}
Since $S$ has subexponential growth it cannot have free subsemigroups of
rank greater than $1$, so Corollary~\ref{cor_klawedichotomy} tells us
that $S$ is left reversible. Thus, the quotient semigroup $S' = S / \mathord{\cong}$
is well-defined and right cancellative by \cite[Proposition~1.24]{Paterson88},
and left cancellative by Proposition~\ref{prop_klawecanc}.

It is immediate that the quotient $S'$ is finitely generated with growth
function bounded above by that of $S$ (hence in particular subexponential).
Hence, by Theorem~\ref{thm_polyfc}, $S'$
satisfies FC. But since $S'$ is left cancellative,
this means $S'$ it satisfies SFC, which by \cite[Theorems 1 and 5]{Argabright67}
means that $S$ is left amenable and satisfies SFC.
\end{proof}



We also recover from Theorem~\ref{thm_polyklawe} a new elementary proof of the
well-known fact that commutative semigroups are amenable (or equivalently,
satisfy SFC).

\begin{corollary}[{\cite[Theorem~4]{Argabright67}}]
All commutative semigroups are amenable.
\end{corollary}
\begin{proof}
Commutative semigroups trivially satisfy the
Klawe condition, and finitely generated ones clearly have
growth functions bounded above by a polynomial of degree the 
number of generators. Hence, given a commutative semigroup $S$, Theorem~\ref{thm_polyklawe}
tells us that all of its finitely generated subsemigroups are
left amenable, which by \cite[Problem 0.30]{Paterson88} suffices for
$S$ to be left amenable.

Finally, right amenability is easily observed to be trivially equivalent
to left amenability in the commutative case.
\end{proof}

We have shown that subexponential growth suffices for left amenability
in a large class of finitely generated semigroups. We know it does not
suffice for semigroups in absolute generality, because not all finite
semigroups are left amenable \cite[Corollary 1.19]{Paterson88}. We do not know if lack of
left reversibility is the only possible obstruction here, that is, whether
a left reversible, finitely generated semigroup of subexponential growth
is necessarily left amenable.  We conjecture that it is not, even if
subexponential growth is replaced by the stronger condition of polynomial
growth:

\begin{conjecture}
There is a left reversible, finitely generated semigroup of polynomial
growth which is not left amenable.
\end{conjecture}

Just as for the Sorenson conjecture, the question (for both polynomial
and subexponential cases) reduces to the case of semigroups which are
right cancellative but not left cancellative. Indeed, if $S$ is an
example as postulated by the conjecture (or a corresponding example with
subexponential in place of polynomial growth) then $S' = S / \mathord{\cong}$ is well defined
and right cancellative by \cite[Proposition~1.24]{Paterson88}. Moreover,
$S'$ is finitely generated, of polynomial (or subexponential) growth, left
reversible (since this property is clearly inherited by quotients) and
also not left amenable \cite[Proposition~1.25]{Paterson88}. Finally, $S'$ is not left cancellative,
since if it were then by Proposition~\ref{prop_klawecanc} $S$ would 
satisfy the Klawe condition, but then by Theorem~\ref{thm_polyklawe} $S$
would be left amenable, giving a contradiction.


\section{Quasi-Isometry Invariance of F\o lner Conditions}

In this section we explore the extent to which amenability and the
F\o lner conditions are ``geometric'' properties of finitely generated
semigroups, in the sense of being invariant under left and/or right
quasi-isometry.

\subsection{Digraphs, Quasi-Isometry and Semigroups}
If $\Gamma$ is a digraph we write $V\Gamma$ for its vertex
set and $E \Gamma$ for its edge set, which we view as a subset of
$V\Gamma \times V \Gamma$. We write $P_f(X)$ for the set of finite
subsets of a set $X$. We write $\mathbb{\overline{R}}$ for the set of
non-negative real numbers with $\infty$ adjoined, and extend multiplication,
addition and the usual order on $\mathbb{R}$ to $\mathbb{\overline{R}}$
in the obvious way, leaving $0 \infty$ undefined.

We briefly recall the definition of quasi-isometry for digraphs
and semigroups --- see \cite{K_semimetric} for a detailed introduction.
Given a digraph $\Gamma$, we define a \textit{semimetric}
$d_\Gamma : V\Gamma \times V\Gamma \to \mathbb{\overline{R}}$
by setting $d_\Gamma(x,y)$ to be the length of the shortest directed path
from $x$ to $y$ in $\Gamma$,
or $\infty$ if there is no such path. We say that two digraphs $\Gamma$ and
$\Delta$ are 
\textit{quasi-isometric} if there is a map $\phi : V\Gamma \to V \Delta$
and a real number $\lambda > 0$ such that for all vertices $x$ and $y$ of $\Gamma$ we have
$$\frac{1}{\lambda} d_\Delta(\phi(x),\phi(y)) - \lambda \ \leq \ d_\Gamma(x,y) \ \leq \ \lambda d_\Delta(\phi(x),\phi(y)) + \lambda,$$
and for every vertex $y$ of $\Delta$ there is a vertex $x$ of $\Gamma$
with $d_\Delta(\phi(x),y) \leq \lambda$ and $d_\Delta(y, \phi(x)) \leq \lambda$. The map
$\phi$ is called a \textit{$\lambda$-quasi-isometry}. Quasi-isometry gives
rise to an equivalence relation on the class of all digraphs
\cite[Proposition~1]{K_semimetric}.

If $S$ is a semigroup with a finite generating set $X$, the
\textit{right} [respectively, \textit{left}] \textit{Cayley graph of $S$ with respect
to $X$}, denoted $\Gamma_r(S,X)$ [respectively, $\Gamma_l(S,X)$], is the digraph
with vertex set $S$ and an edge from $s$ to $t$ if and only if $sx = t$
[respectively, $xs = t$]
for some $x \in X$. Although the right (or left) Cayley graph depends on the choice of
generators, its quasi-isometry class does not \cite[Proposition~4]{K_semimetric};
thus we may say that two finitely generated semigroups are
\textit{right} [\textit{left}] \textit{quasi-isometric} if their right [left]
Cayley graphs are quasi-isometric, without concerning ourselves about
choice of generators.

Unlike in a group, the left and right Cayley graphs of a semigroup need not
be isomorphic,
or indeed even quasi-isometric. The left and right quasi-isometry classes of a semigroup thus
form two distinct natural invariants which in some
sense capture its ``large-scale geometry'', and it is natural to ask
what properties of a finitely generated semigroup can be seen in either
or both of these quasi-isometry classes.
Our first observation is that the key property of left reversibility
can be seen in one of the quasi-isometry classes (although not the other):
\begin{proposition}\label{prop_leftrevqsi}
Left reversibility is a right (but not a left) quasi-isometry invariant
of semigroups.
\end{proposition}
\begin{proof}
From the definition it follows immediately that a semigroup is left
reversible if and only if every pair of principal right ideals contain
a common principal right ideal. Thus, left reversibility can be seen
in the structure of the containment order on the set of principal right
ideals. By \cite[Proposition 5]{K_semimetric}, this order is a right
quasi-isometry invariant of semigroups.

On the other hand, the $2$-element left-zero semigroup (with both elements
as generators) is not left reversible, but its left Cayley graph is a
complete digraph on 2 vertices, so isomorphic (and in particular
quasi-isometric) to that of $\mathbb{Z}_2$, which is left reversible,
with both elements as generators.
\end{proof}

\subsection{Isoperimetric Numbers of Digraphs} Let $\Gamma$ be a digraph.
For any subset $A$ of $V\Gamma$ we define the \textit{out-boundary}
$\partial A$ of $A$ to be the set
\[
\partial A = 
\{
x \in V\Gamma : \exists (y,x) \in ED \ \mbox{with} \  y \in A \ \textrm{ and } \ x \notin A
\}. 
\]
Now we define the (\textit{outward}) \textit{isoperimetric number} of
$\iota(\Gamma)$ of $\Gamma$ by:
\[
\iota(\Gamma) = \inf_{A \in P_f(V\Gamma)}\left(  \frac{|\partial A|}{|A|} \right).
\]
We shall be interested in digraphs with isoperimetric number zero, that is,
graphs $\Gamma$ for which
\begin{equation}
(\forall \epsilon > 0) \; (\exists A \in \mathcal{P}_f(V\Gamma)) \; : \; \frac{|\partial A|}{|A|} < \epsilon. 
\label{eq_star}
\end{equation}
%
%
%
%
%
We note that the notions of outboundary and isoperimetric number of a digraph defined here arise in the literature in the definitions of Cheeger constants and Cheeger inequalities as part of the spectral theory of directed graphs; see \cite{Fan2005}.

The following result shows the relationship between FC and isoperimetric number. 
\begin{proposition}\label{prop_isop}
Let $S$ be a semigroup generated by a finite set $X$. Then $S$ satisfies
FC if and only if the left Cayley graph
$\Gamma_l(S,X)$ has isoperimetric number zero. 
\end{proposition}
\begin{proof}
Suppose first that $S$ satisfies FC.
We must show that $\Gamma$ satisfies \eqref{eq_star}. Let $\epsilon > 0$
be given and choose  $\delta > 0$ with $\delta < \epsilon / |X|$. Since $S$
satisfies FC and $X$ is a finite set, we may choose a finite subset $F$
of $S$ such that
\[
(\forall x \in X)(|xF \setminus F| < \delta |F|).
\] 
We claim that by taking $A = F$  condition \eqref{eq_star} will be satisfied.
Indeed, by the definition of $\Gamma$ we have 
\[
\partial A = \bigcup_{x \in X} (xA \setminus A) =  \bigcup_{x \in X} (xF \setminus F),
\]
and so
\begin{align*}
|\partial A|
\ & = \ \left| \bigcup_{x \in X} (xF \setminus F) \right| \ \leq \ \sum_{x \in X} |xF \setminus F| \\
& < \ |X| \delta|F| \ < \ |X| (\epsilon / |X|) |F| \ = \ \epsilon |F| \ = \ \epsilon |A|, 
\end{align*}
which shows that \eqref{eq_star} holds. 

Conversely, suppose that $\Gamma$ has isoperimetric number zero or,
equivalently, suppose that \eqref{eq_star} holds. We must show that $S$
satisfies FC.
To this end, let $H \in \mathcal{P}_f(S)$ and $\epsilon > 0$ be given. Let
$r \in \mathbb{N}$ be such that every element $h \in H$ can be
written as a product of elements from $X$ of length at most $r$; such
an $r$ exists because $H$ is finite and $X$ is a generating set for
$S$. Write $X^{\leq r}$ for the set of elements of $S$ which can be
written as a product of elements of $X$ of length at most $r$. Set
$C = |X^{\leq r}|$ and $\delta = \epsilon / C$.
By \eqref{eq_star} we may choose a finite set $F \in \mathcal{P}_f(S)$
such that $|\partial F| / |F| < \delta$.

Let $h \in H$ be arbitrary and consider the set $hF \setminus F$. Write
$h = x_1 x_2 \ldots x_t$ where $x_i \in X$ and $t \leq r$. This is
possible by the definition of $r$. Consider an arbitrary element of
$hF \setminus F$, say $hf$ where $f \in F$. Then there is a path of
length $t$ in the
left Cayley graph from $f$ to $hf = x_1 x_2 \ldots x_t f$ with the edges
labelled $x_1, x_2, \cdots, x_t$. Since this path leads from a vertex
in $F$ to a vertex not in $F$, it must have at least one vertex in
$\partial F$. It follows that $hf$ can be written
as the product of an element of $X^{\leq r}$ and an element of $\partial F$.
Thus, there are at most $C |\partial F|$ choices of $hf \in hF \setminus F$,
so we have
\[
|hF \setminus F|
\ \leq \ C |\partial F|
\ < \ C \left( \delta |F| \right)
\ < \ C \left( \frac{\epsilon |F| }{C} \right)
\ = \ \epsilon |F|. 
\]
This holds for all $h \in H$ and therefore $S$ satisfies FC. 
\end{proof}

Next we shall show that the property of having isoperimetric number zero is a quasi-isometry invariant of directed graphs with bounded out-degree. We first need a lemma. 

\begin{lemma}\label{lem_const}
Let $\phi: V \Gamma \rightarrow V \Delta$ be a quasi-isometry between
digraphs. If $\Gamma$ has bounded out-degree then $\phi$ has bounded
fibre sizes.
\end{lemma}
\begin{proof}
Let $k$ be an upper bound on the out-degree of $\Gamma$, and suppose
$\phi$ is a $\lambda$-quasi-isometry. Fix an element $x \in V \Gamma$ and
consider the set of elements $y \in V \Gamma$ with $\phi(y) = \phi(x)$.
For each such $y$ we have
$$d_\Gamma(x,y) \ \leq \ \lambda d_\Delta(\phi(x),\phi(y)) + \lambda \ = \ \lambda d(\phi(x),\phi(x)) + \lambda \ = \ \lambda.$$
Thus, every such $y$ is reachable by a directed path of length
at most $\lambda$ starting from $x$. But there are clearly no more than
$$\sum_{i=0}^{\lambda'} k^i$$
such paths, where $\lambda'$ is the integer part of $\lambda$, so this number, which
is independent of $x$, forms a bound on the fibre size.
\end{proof}

\begin{figure}
\begin{center}
\scalebox{1.5}
{
\begin{tikzpicture}
[scale=0.9, 
decoration={ 
markings,
mark=
at position 0.6 
with 
{ 
\arrow[scale=1]{stealth} 
} 
},
SLball/.style ={circle, draw=black!80, thick, fill=gray!15, minimum height=4em, opacity=0.6},
SRball/.style ={circle, draw=black, thick, fill=gray!15, minimum height=5em, opacity=0.7},
Sball/.style ={circle, draw=black!60, thick, fill=gray!15, minimum height=3em, opacity=0.25},
SSball/.style ={circle, draw=black!60, thick, fill=gray!15, minimum height=3em, opacity=.8}
],
\tikzstyle{vertex}=[circle,draw=black, fill=black, inner sep = 0.3mm]
\node (bLeft) at (0em,0em) [SLball] {};
\node (bRight) at (10em,0em) [SRball] {};
\node (b2) at (10em,2em) [Sball] {};
\node (b3) at (10em,-2em) [Sball] {};
\node (b4) at (12em,0em) [SSball] {};
\node (b5) at (8em,0em) [Sball] {};
\node (b6) at (11.5em,1.5em) [Sball] {};
\node (b7) at (8.5em,-1.5em) [Sball] {};
\node (b8) at (11.5em,-1.5em) [Sball] {};
\node (b9) at (8.5em,1.5em) [Sball] {};
\node (b1) at (10em,0em) [Sball] {};
\node (x)  at (1,0) {};
\node (y)  at (2.4,0) {};
\arrowdraw{x}{y}
\node (Gamma) at (0em,-6em) {\tiny{$\Gamma$}};
\node (Delta) at (10em,-6em) {\tiny{$\Delta$}};
\node (P) at (1.7,0.3) {\tiny{$\phi$}};
\node (A) at (0,0) {\tiny{$A$}};
\node (Q) at (12em,-4em) {\tiny{$Q$}};
\node (Ball) at (14em,-2em) {\tiny{$\mathcal{B}_\lambda(\phi(a))$}};
\node (PA) at (7em,-2em) {\tiny{$\phi(A)$}};
\node (a)  at (-1.3em,-0.8em) {\tiny{$a$}};
\node (b)  at (-1em,3.5em) {\tiny{$\partial A \ni b $ \; \; \; \; \ }};
\node (t)  at (4.2em,5em) {\tiny{$t$}};
\node (leq)  at (1.6em,3em) {\tiny{$\leq 2\lambda^2 + 2\lambda$}};
\node (Pa) at (12em,0em) {\tiny{$\phi(a)$}};
\node (y)  at (14em,2.5em) {\tiny{\; \; \;  \;  \ $y \in \partial Q$ }};
\node (z)  at (18em,-2em) {\tiny{\; \;  \ \ $z = \phi(t)$ }};
\node (leq2)  at (16.6em,0.6em) {\tiny{$\leq \lambda$}};
\draw[>=stealth, ->, snake=snake,segment amplitude = .4mm, segment length = 3mm] (a) -- (b);
\draw[>=stealth, ->, snake=snake,segment amplitude = .4mm, segment length = 4mm] (-0.7em,3.5em) -- (t);
\draw[>=stealth, ->, snake=snake,segment amplitude = .4mm, segment length = 3mm] (12.5em,0.5em) --  (14em,2em);
\draw[>=stealth, <->, snake=snake,segment amplitude = .4mm, segment length = 3mm] (14em,2em) -- (17.6em,-1.6em);
%
%
%
%
%
%
%
%
%
\end{tikzpicture}
}
\end{center}
\caption{Diagram showing the inequality $|\partial Q| \leq DE|\partial A|$ in the proof of Theorem~\ref{thm_digraph_QSI}. 
}\label{fig_qsi}
\end{figure}
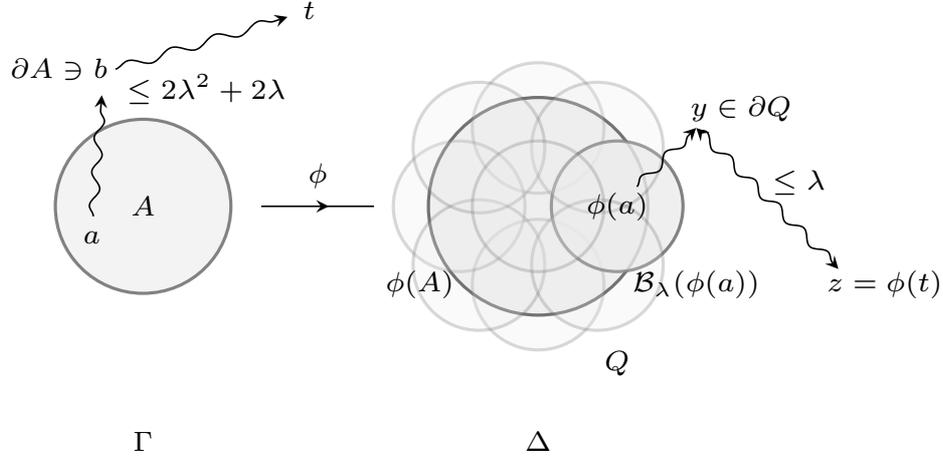

\begin{theorem}\label{thm_digraph_QSI}
Let $\Gamma$ and $\Delta$ be directed graphs both with bounded out-degree. 
If $\Gamma$ and $\Delta$ are quasi-isometric then $\iota(\Gamma)=0$ if and only if $\iota(\Delta)=0$. 
\end{theorem}
\begin{proof}
Let $\phi: \Gamma \rightarrow \Delta$ be a $\lambda$-quasi-isometry. Let
$C$ be the bound on the fibre size of $\phi$ given by Lemma~\ref{lem_const}.
Since $\Delta$ has bounded out-degree, we may choose an upper bound (call
it $E$) on the number of directed paths in $\Delta$ originating at any one
vertex and having length at most $\lambda$. Similarly, let $D$ be an
upper bound on the number of directed paths in $\Gamma$ originating at any
one vertex and having length at most $2 \lambda^2 + 2 \lambda$.

Suppose that $\iota(\Gamma)=0$. We claim that $\iota(\Delta)=0$, so, we must show that in the digraph $\Delta$ we have
\[
(\forall \epsilon > 0) \; (\exists Q \in \mathcal{P}_f(V\Delta)) \; : \; \frac{|\partial Q|}{|Q|} < \epsilon. 
\]
Let $\epsilon > 0$ be given. 
Since $\iota(\Gamma)=0$ we may choose a finite subset $A$ of $V\Gamma$
such that
\begin{equation}\label{eq_zero}
\frac{|\partial A|}{|A|} \ < \ \frac{\epsilon}{CDE}.
\end{equation}
Set
\[
Q \ = \ \bigcup_{a \in A} \mathcal{B}_{\lambda} (\phi(a)) \ \subseteq \ V\Delta,
\]
where
$$\mathcal{B}_{\lambda}(x) \ = \ \lbrace a \in V\Delta \mid d(a,x) \leq \lambda \textrm{ and } d(x,a) \leq \lambda \rbrace.$$
We claim that $\frac{|\partial Q|}{|Q|} < \epsilon$ and the rest of the proof will be devoted to establishing this fact. 

Firstly, since $\phi(A) \subseteq Q$, by Lemma~\ref{lem_const} we have
\begin{equation}\label{eq_one}
|Q| \ \geq \ |\phi(A)| \ \geq \ \frac{|A|}{C}.
\end{equation}
Now we want to estimate the size of the set $\partial Q$, seeking an upper bound in terms of $|\partial A|$.  

In $\Delta$ let $y \in \partial Q$ be an arbitrary element of the out-boundary of the set $Q$. 
Since $\phi$ is a $\lambda$-quasi-isometry there is a vertex
$z \in \phi(V\Gamma)$ such that
\begin{equation}\label{eq_nice}
d_{\Delta}(z,y) \leq \lambda  \textrm{ and } d_{\Delta}(y,z) \leq \lambda. 
\end{equation}
Let $t \in V \Gamma$ be such that $\phi(t) = z$.
Notice that $t \notin A$; indeed,
 $y \in \mathcal{B}_{\lambda} (z) = \mathcal{B}_{\lambda} (\phi(t))$ so
if $t$ were in $A$ then, by the definition of $Q$, we would have
$y \in Q$ contradicting $y \in \partial Q$. 

Now by the definitions of $Q$ and $\partial Q$, there exists $a \in A$
such that
\[
d_{\Delta}(\phi(a), y) \ \leq \ \lambda + 1. 
\]
Together with \eqref{eq_nice} and the directed triangle inequality this gives
\[
d_{\Delta}(\phi(a), \phi(t)) \ = \ d_{\Delta}(\phi(a), z) \ \leq \ 2\lambda + 1. 
\]
Applying the quasi-isometry inequality to this, we obtain
\[
d_{\Gamma}(a,t) \ \leq \ \lambda \ d_{\Delta}(\phi(a), \phi(t)) + \lambda \ \leq \ \lambda(2\lambda + 1) + \lambda \ = \ 2\lambda^2 + 2\lambda. 
\]
So in the digraph $\Gamma$ we have $a \in A$, $t \not\in A$, and there is a
directed path from $a$ to $t$ in $\Gamma$ of length at most 
$2\lambda^2 + 2\lambda$. Any such path must include at least one vertex
from the out-boundary $\partial A$. It follows that there is a directed
path of length at most $2\lambda^2 + 2\lambda$ from some vertex
$b \in \partial A$ to the vertex $t$. The number of possibilities for
$b$ is $|\partial A|$, and once $b \in \partial A$ is chosen, there are
at most $D$ possible paths of the given length. We conclude that the
number of possibilities for the vertex $t$ is bounded above by
$D |\partial A|$. Since $z = \phi(t)$ it follows that the number of
distinct vertices $z$ that can occur in the above argument is also
bounded above by $D |\partial A|$. But every $y \in \partial Q$ is
reachable from some such $z$ by a path of length at most $\lambda$, and
there are at most $E$ such paths from each $z$, so we obtain
\begin{equation}\label{eq_two}
|\partial Q| \ \leq \ DE |\partial A|. 
\end{equation}
Finally, combining equations \eqref{eq_zero}, \eqref{eq_one} and \eqref{eq_two} 
we obtain
\[
\frac{|\partial Q|}{|Q|} \ \leq \ \left( \frac{C}{|A|} \right) \left( D E |\partial A| \right) \ < \ \epsilon
\]
as required to complete the proof. 
\end{proof}

An immediate consequence of the above results is the following. 

\begin{theorem}\label{thm_qsi}
The F\o lner condition FC is a left Cayley graph quasi-isometry invariant of finitely generated semigroups. 
\end{theorem}
\begin{proof}
This follows from Proposition~\ref{prop_isop} and Theorem~\ref{thm_digraph_QSI}. 
\end{proof}

For left cancellative semigroups, where the conditions of left amenability,
FC and SFC coincide, this yields a direct generalisation of a well-known
result for groups:
\begin{corollary}\label{cor_AM_QSI}
Left amenability is a left quasi-isometry invariant of finitely generated
left cancellative semigroups. 
\end{corollary}

In contrast, the fact that left reversibility is not visible in the left
Cayley graph (Proposition~\ref{prop_leftrevqsi}) means that neither
SFC nor left amenability are left quasi-isometry invariants for
finitely generated semigroups in general. Indeed, we have already seen
that the
$2$-element left zero semigroup, which is not left reversible (and hence
not left amenable), is left quasi-isometric to the amenable (and hence SFC)
group $\mathbb{Z}_2$. However, since left reversibility is a right
quasi-isometry invariant (Proposition~\ref{prop_leftrevqsi}),
it is natural to ask whether SFC and/or left amenability are invariants
of the right quasi-isometry class, or of the two quasi-isometry classes
considered together.

\begin{question}
Can SFC and/or left amenability of a finitely generated semigroup be determined from (i)
the right quasi-isometry class or (ii) the left and right quasi-isometry
classes together?
\end{question}

\bibliographystyle{plain}

\def\cprime{$'$} \def\cprime{$'$}


\begin{thebibliography}{10}

\bibitem{Argabright67}
L.~N. Argabright and C.~O. Wilde.
\newblock Semigroups satisfying a strong {F}\o lner condition.
\newblock {\em Proc. Amer. Math. Soc.}, 18:587--591, 1967.

\bibitem{Cecc10}
T.~Ceccherini-Silberstein and M.~Coornaert.
\newblock {\em Cellular automata and groups}.
\newblock Springer Monographs in Mathematics. Springer-Verlag, Berlin, 2010.

\bibitem{Cecc15}
T.~Ceccherini-Silberstein and M.~Coornaert.
\newblock The {M}yhill property for cellular automata on amenable semigroups.
\newblock {\em Proc. Amer. Math. Soc.}, 143(1):327--339, 2015.

\bibitem{Cedo07}
F.~Ced{\'o} and J.~Okni{\'n}ski.
\newblock Semigroups of matrices of intermediate growth.
\newblock {\em Adv. Math.}, 212(2):669--691, 2007.

\bibitem{Fan2005}
Fan Chung.
\newblock Laplacians and the {C}heeger inequality for directed graphs.
\newblock {\em Ann. Comb.}, 9(1):1--19, 2005.

\bibitem{Day57}
M.~M. Day.
\newblock Amenable semigroups.
\newblock {\em Illinois J. Math.}, 1:509--544, 1957.

\bibitem{Donnelly13}
J.~Donnelly.
\newblock Necessary and sufficient conditions for a subsemigroup of a
  cancellative left amenable semigroup to be left amenable.
\newblock {\em Int. J. Algebra}, 7(5-8):237--243, 2013.

\bibitem{Ghys89}
{\'E}.~Ghys and P.~de~la Harpe.
\newblock Infinite groups as geometric objects (after {G}romov).
\newblock In {\em Ergodic theory, symbolic dynamics, and hyperbolic spaces
  ({T}rieste, 1989)}, Oxford Sci. Publ., pages 299--314. Oxford Univ. Press,
  New York, 1991.

\bibitem{K_svarc}
R.~D. Gray and M.~Kambites.
\newblock A \v{S}varc-{M}ilnor lemma for monoids acting by isometric
  embeddings.
\newblock {\em Internat. J. Algebra Comput.}, 21:1135--1147, 2011.

\bibitem{K_semimetric}
R.~D. Gray and M.~Kambites.
\newblock Groups acting on semimetric spaces and quasi-isometries of monoids.
\newblock {\em Trans. Amer. Math. Soc.}, 365:555--578, 2013.

\bibitem{K_qsifp}
R.~D. Gray and M.~Kambites.
\newblock Quasi-isometry and finite presentations for left cancellative
  monoids.
\newblock {\em Internat. J. Algebra Comput.}, 23:1099--1114, 2013.

\bibitem{K_hyper}
R.~D. Gray and M.~Kambites.
\newblock A strong geometric hyperbolicity property for directed graphs and
  monoids.
\newblock {\em J. Algebra}, 420:373--401, 2014.

\bibitem{Grigorchuk80}
R.~I. Grigorchuk.
\newblock Symmetrical random walks on discrete groups.
\newblock In {\em Multicomponent random systems}, volume~6 of {\em Adv. Probab.
  Related Topics}, pages 285--325. Dekker, New York, 1980.

\bibitem{Grigorchuk88}
R.~I. Grigorchuk.
\newblock Semigroups with cancellations of degree growth.
\newblock {\em Mat. Zametki}, 43(3):305--319, 428, 1988.

\bibitem{Klawe77}
M.~Klawe.
\newblock Semidirect product of semigroups in relation to amenability,
  cancellation properties, and strong {F}\o lner conditions.
\newblock {\em Pacific J. Math.}, 73(1):91--106, 1977.

\bibitem{Paterson88}
A.~L.~T. Paterson.
\newblock {\em Amenability}, volume~29 of {\em Mathematical Surveys and
  Monographs}.
\newblock American Mathematical Society, Providence, RI, 1988.

\bibitem{Shneerson08}
L.~M. Shneerson.
\newblock Polynomial growth in semigroup varieties.
\newblock {\em J. Algebra}, 320(6):2218--2279, 2008.

\bibitem{Sorenson64}
J.~R. Sorenson.
\newblock Left-amenable semigroups and cancellation.
\newblock {\em Notices of Amer. Math. Soc.}, 11:763, 1964.

\bibitem{Sorenson66}
J.~R. Sorenson.
\newblock {\em Existence of measures that are invariant under a semigroup of
  transformations}.
\newblock PhD thesis, Purdue University, 1966.

\bibitem{Takahashi03}
Y.~Takahashi.
\newblock A counterexample to {S}orenson's conjecture: the finitely generated
  case.
\newblock {\em Semigroup Forum}, 67(3):429--431, 2003.

\bibitem{vonNeumann29}
J.~von Neumann.
\newblock Zur allgemeinen theorie des ma{\ss}es.
\newblock {\em Fund. Math.}, 13:73--111, 1929.

\bibitem{Willson09}
B.~Willson.
\newblock Reiter nets for semidirect products of amenable groups and
  semigroups.
\newblock {\em Proc. Amer. Math. Soc.}, 137(11):3823--3832, 2009.

\bibitem{Yang87}
Zhuocheng Yang.
\newblock F\o lner numbers and {F}\o lner type conditions for amenable
  semigroups.
\newblock {\em Illinois J. Math.}, 31(3):496--517, 1987.

\end{thebibliography}
\end{document}